\documentclass[review]{elsarticle}

\usepackage{lineno,hyperref}
\modulolinenumbers[5]
\usepackage{epsfig}
\usepackage{geometry}
\usepackage{color}
\usepackage{amsmath}
\usepackage{amssymb}
\usepackage{amsthm}
\usepackage{amsfonts}

\newtheorem{defn}{Definition}[section]
\newtheorem{lem}{Lemma}[section]
\newtheorem{thm}{Theorem}[section]
\newtheorem{rem}{Remark}[section]
\newtheorem{corl}{Corollary}[section]

\journal{arXiv}

\begin{document}

\begin{frontmatter}

\title{Stability analysis for a class of nonlinear  time-changed systems}
%\tnotetext[mytitlenote]{Fully documented templates are available in the elsarticle package on %\href{http://www.ctan.org/tex-archive/macros/latex/contrib/elsarticle}{CTAN}.}

%% Group authors per affiliation:
%\author{Qiong Wu}
%\address{Tufts University}
%\fntext[myfootnote]{Since 1880.}

%% or include affiliations in footnotes:
%\author[mymainaddress,mysecondaryaddress]{Qiong Wu}
%\ead[url]{www.elsevier.com}

\author{Qiong Wu\corref{mycorrespondingauthor}}
\cortext[mycorrespondingauthor]{Corresponding author}
\ead{Qiong.Wu@tufts.edu}

%\address[mymainaddress]{1600 John F Kennedy Boulevard, Philadelphia}
\address{Tufts University, \\Department of Mathematics, 503 Boston Avenue, Medford, MA 02155, USA.}

\begin{abstract}
This paper investigates the stability of a class of differential systems time-changed by $E_{t}$ which is the inverse of a $\beta$-stable subordinator. In order to explore stability, a time-changed Gronwall's inequality and a generalized It\^o formula related to both the natural time $t$ and the time-change $E_{t}$ are developed. For different time-changed systems, corresponding stability behaviors such as exponential sample-path stability, $p$th moment asymptotic stability and $p$th moment exponential stability are investigated. Also a connection between the stability of the  time-changed system and that of its corresponding non-time-changed system is revealed. 
\end{abstract}


\end{frontmatter}

\linenumbers

\section{Introduction}
Linear and nonlinear systems play an important role in applied areas, for example, control theory, mathematical biology and convex optimization. The stability of linear and nonlinear systems is extensively discussed in \cite{Rugh1996, Feng1992}. Focusing on delay phenomena in the natural sciences, the delayed linear and nonlinear systems are developed and the stability analysis is performed in \cite{erneux2009applied}. Fractional systems can be used to describe complex phenomena in engineering. Various kinds of stabilities of linear and nonlinear fractional dynamic systems are discussed in \cite{matignon1996stability}. More recently,
the following time-changed differential systems are studied in \cite{Kei2011}, 
\begin{eqnarray}\label{generaltimechangedSDE}
\mathrm{d}X(t) = \rho(t, X(t))\mathrm{d}t + \mu(E_{t}, X(t))\mathrm{d}E_{t} + \delta(E_{t}, X(t))\mathrm{d}B_{E_{t}}, ~~~X(0) = x_{0}\in\mathbb{R}^{d}.
\end{eqnarray}
where  $E_{t}$ is a random time-change denoting a new clock. For instance, $E_{t}$ might represents the business time at the calendar time $t$. Specifically, $E_{t}$ is considered as the general inverse of a $\beta$-stable subordinator $U(t)$, defined as
\begin{eqnarray}\label{defofE_{t}}
E_{t} = \inf\{s > 0: U(s) > t\},
\end{eqnarray}
where the stable subordinator $U(t)$ with index $\beta\in(0, 1)$ is a strictly increasing $\beta$-stable L\'evy process and takes Laplace transform
\begin{eqnarray*}
	\mathbb{E}[\exp(-sU(t))] = \exp(-ts^{\beta}).
\end{eqnarray*}
In particular, $E_{t}$ is a continuous time-change since $U(t)$ is strictly increasing. For more details on $\beta$-stable L\'evy processes and their inverses, please see \cite{Janicki1994}. To our best knowledge, there are no results on the stability of any kinds of time-changed differential systems.  In this paper,  the stabilities of various kinds of time-changed differential systems are discussed based on developing a Gronwall's inequality and generalized It\^o formula.  
\section{Preliminaries}\label{Pre}
In this section, several helpful lemmas and definitions are introduced  to illustrate the main stability results to be considered later. Lemma~\ref{finitevaraitionprocessissemimartingale} below indicates that the time-change $E_{t}$ is a semimartingale.
\begin{lem}\label{finitevaraitionprocessissemimartingale}
(\cite{Mircea2002}) If $X_{t}$ is an adapted process with c\`{adl\`{a}g} paths of finite variation on compacts, then $X_{t}$ is a semimartingale.
\end{lem}
Let $B_{t}$ be a standard Brownian motion and $E_{t}$ be the time-change. Consider the following filtration $\mathcal{F}_{t}$ generated by $B_{t}$ and $E_{t}$
\begin{eqnarray}\label{filtrationgeneratedbyBandE}
	\mathcal{F}_{t} = \bigcap_{u>t}\big\{\sigma(B_{s}: 0\leq s\leq u)\vee\sigma(E_{s}: s\geq 0)\big\},
\end{eqnarray}
where $\sigma_{1}\vee\sigma_{2}$ denotes the $\sigma$-field generated by the union $\sigma_{1}\cup\sigma_{2}$ of $\sigma$-fields $\sigma_{1}, \sigma_{2}$.
\begin{lem}\label{timechangedBrown}
	(\cite{Marcin2010})The time-changed Brownian motion, $B_{E_{t}}$, is a square integrable martingale with respect to the filtration $\{\mathcal{F}_{E_{t}}\}_{t\geq 0}$, where $\{\mathcal{F}_{t}\}$ is the filtration given in Eq.~\eqref{filtrationgeneratedbyBandE}. The quadratic variation of the time-changed Brownian motion satisfies $\langle B_{E_{t}}, B_{E_{t}}\rangle = E_{t}$.
\end{lem}
From Lemmas~\ref{finitevaraitionprocessissemimartingale} and~\ref{timechangedBrown}, it is well known that integrals with respect to the time-change, $E_{t}$, and the time-changed Brownian motion, $B_{E_{t}}$, are well-defined. Moreover, the following two lemmas provide connections among different kinds of time-changed integrals.
%\begin{lem}
%	(\cite{Kei2014}Lemma 8)Let $E_{t}$ be the inverse of a %$\beta-$stable subordinator $W(u)$ with $\beta\in(0, 1)$, then %moments of $E_{t}$ of all orders exist and are given by 
%	\begin{eqnarray*}
%		\mathbb{E}(E^{n}_{t}) = \frac{n!}{\Gamma(n\beta + 1)}t^{n\beta}
%	\end{eqnarray*}
%	where $\Gamma(\cdot)$ is the Gamma function.
%\end{lem}

\begin{lem}\label{1stchangeofvariable}
(1st Change-of-Variable Formula\cite{Kei2011}, \cite{Jean1979}) Let $E_{t}$ be the $(\mathcal{F}_{t})$-measurable time-change. Suppose $\mu(t)$ and $\delta(t)$ are $(\mathcal{F}_{t})$-measurable and integrable. Then, for all $t\geq 0$ with probability one,
\begin{eqnarray*}
	\int_{0}^{E_{t}}\mu(s)\mathrm{d}s + \int_{0}^{E_{t}}\delta(s)\mathrm{d}B_{s} = \int_{0}^{t}\mu(E_{s})\mathrm{d}E_{s} + \int_{0}^{t}\delta(E_{s})\mathrm{d}B_{E_{s}}.
\end{eqnarray*}
\end{lem}

\begin{lem}\label{2ndchangeofvariable}
	(2nd Change-of-Variable Formula\cite{Kei2011}) Let $E_{t}$ be the $(\mathcal{F}_{t})$-measurable time-change which is the general inverse $\beta$-stable subordinator $U(t)$. Suppose $\mu(t)$ and $\delta(t)$ are $(\mathcal{F}_{t})$-measurable and integrable. Then, for all $t\geq 0$ with probability one,
	\begin{eqnarray*}
		\int_{0}^{t}\mu(s)\mathrm{d}E_{s} + \int_{0}^{t}\delta(s)\mathrm{d}B_{E_{s}} = \int_{0}^{E_{t}}\mu(U(s-))\mathrm{d}s + \int_{0}^{E_{t}}\delta(U(s-))\mathrm{d}B_{s}.
	\end{eqnarray*}
\end{lem}
The next lemma reveals a deep connection between the time-changed SDE~\eqref{timechangedSDE} and its corresponding classical non-time-changed SDE~\eqref{classicSDE}.
\begin{align}
&\mathrm{d}X(t) = \mu(E_{t}, X(t))\mathrm{d}E_{t} + \delta(E_{t}, X(t))\mathrm{d}B_{E_{t}}, ~~~X(0) = x_{0};\label{timechangedSDE}\\
&\mathrm{d}Y(t) = \mu(t, Y(t))\mathrm{d}t + \delta(t, Y(t))\mathrm{d}B_{t}, ~~~Y(0) = x_{0};\label{classicSDE}
\end{align}
\begin{lem}\label{duality}
(\cite{Kei2011} Duality) Let $E_{t}$ be the inverse of a $\beta$-stable subordinator $U(t)$.
\begin{enumerate}
	\item If a process $Y(t)$ satisfies the SDE~\eqref{classicSDE}, then the process $X(t):= Y(E_{t})$ satisfies the SDE~\eqref{timechangedSDE}.
	\item If a process $X(t)$ satisfies the SDE~\eqref{timechangedSDE}, then the process $Y(t) := X(U(t-))$ satisfies the SDE~\eqref{classicSDE}.
\end{enumerate}
\end{lem}
Without loss of generality, let $X(t) := X(t;x_{0})$ be the solution of the time-changed SDE~\eqref{generaltimechangedSDE} with initial value $x_{0}$. Assume that $\rho(t, 0) = \mu(E_{t}, 0) = \delta(E_{t}, 0) = 0$ for all $t\geq 0$. So SDE~\eqref{generaltimechangedSDE} admits a trivial solution $X(t) \equiv 0$ corresponding to the initial value $x_{0} = 0$. This solution is also called the equilibrium position.
\begin{defn} The trivial solution of SDE~\eqref{generaltimechangedSDE} is said to be
\begin{itemize}
	\item[(1)]	 exponentially sample-path stable if there is a function $\nu(t): [0, \infty)\to[0, \infty)$ approaching $\infty$ as $t\to\infty$ and a pair of positive constants $\lambda$ and $K$ such that for every sample path
	\begin{eqnarray*}
		\|X(t)\rVert \leq K\| x_{0}\rVert\exp(-\lambda \nu(t)),
	\end{eqnarray*}
	where $t\geq 0$ and $x_{0}\in \mathbb{R}^{d}$ is arbitrary;
	\item[(2)] $p$th moment asymptotically stable if there is a function $\nu(t): [0, +\infty)\to[0, \infty)$ decaying to $0$ as $t\to\infty$ and a positive constant $K$ such that 
	\begin{eqnarray*}
		\mathbb{E}\lVert X_{t}( x_{0})\rVert^p \leq K\lVert x_{0}\rVert^p \nu(t)
	\end{eqnarray*}
	for all $t\geq 0$ and $x_{0}\in \mathbb{R}^{d}$;
	\item[(3)] $p$th moment exponentially stable if there is a pair of positive constants $\lambda$ and $K$ such that 
	\begin{eqnarray*}
		\mathbb{E}\lVert X_{t}( x_{0})\rVert^p \leq K\lVert x_{0}\rVert^p\exp(-\lambda t)
	\end{eqnarray*}
	for all $t\geq 0$ and $x_{0}\in \mathbb{R}^{d}$.
\end{itemize}
\end{defn}
\noindent\textbf{Notation:} Assume $A$ is a square matrix. Let $\sigma(A)$ be the spectrum of $A$ and $\mathrm{Re}(\sigma(A))$ be the real part of eigenvalues of $A$.

\section{Stability analysis of time-changed SDEs}
In this section, before investigating the stability of time-changed differential equations, a time-changed Gronwall's inequality is developed and a generalized It\^o formula related to both the natural time and the random time-change is proposed.
\begin{lem}\label{gronwall}
Suppose $U(t)$ is a $\beta$-stable subordinator and $E_{t}$ is the associated inverse stable subordinator. Let $T > 0$ and $x, ~K$: $ \Omega\times [0, T] \to R_{+}$ be $\mathcal{F}_{t}$-measurable functions which are integrable with respect to $E_{t}$. Assume $u_{0}\geq 0$ is a constant. Then, the inequality 
\begin{eqnarray}\label{preinequalityofgronwallinequlity}
	x(t) \leq u_{0} + \int_{0}^{t}K(s)x(s)\mathrm{d}E_{s},  ~~~0\leq t\leq T
\end{eqnarray} 
implies almost surely
\begin{eqnarray*}
	x(t) \leq u_{0}\exp\bigg(\int_{0}^{t}K(s)\mathrm{d}E_{s}\bigg),~~~0\leq t\leq T.
\end{eqnarray*}
\end{lem}
\begin{proof}
Let 
\begin{eqnarray}\label{defofy(t)}
	y(t) := u_{0} + \int_{0}^{t}K(s)x(s)\mathrm{d}E_{s}, ~~~0\leq t\leq T.
\end{eqnarray}
Since $K(s)$ and $x(s)$ are positive, the function $y(t)$ defined in Eq.~\eqref{defofy(t)} is nondecreasing. Moreover, from Eq.s~\eqref{preinequalityofgronwallinequlity} and~\eqref{defofy(t)},  
\begin{eqnarray*}
	x(t)\leq y(t),~~~0 \leq t \leq T,
\end{eqnarray*}
which implies
\begin{eqnarray*}
	y(t) \leq u_{0} + \int_{0}^{t}K(s)y(s)\mathrm{d}E_{s}, ~~0\leq t\leq T.
\end{eqnarray*}
Applying Lemma \ref{2ndchangeofvariable} yields 
\begin{eqnarray}\label{inequalityofapplyingchangeofvariableformula}
	\begin{aligned}
		y(t) &\leq u_{0} + \int_{0}^{E_{t}}K(U(s-))y(U(s-))\mathrm{d}s. 
 	\end{aligned}
\end{eqnarray}
Actually, for $0 \leq t \leq E_{T}$, $U(t-)$ is defined as
\begin{eqnarray*}
	U(t-)  = \inf\{s:s\in [0, T], E_s > t\}\wedge T,
\end{eqnarray*}
which means 
\begin{eqnarray}\label{relationofinverseandsubordinator}
E_{U(t-)} = t~~\textrm{and}~~t \leq U(E_{t}-).
\end{eqnarray} 
Also, let $\tau\in[0, \infty)$ and $\tau\in [0, E_{T}]$, then it holds from Eq.s~\eqref{inequalityofapplyingchangeofvariableformula} and~\eqref{relationofinverseandsubordinator} that
\begin{eqnarray*}
	\begin{aligned}
		y(U(\tau-)) &	\leq u_{0} + \int_{0}^{E_{U(\tau-)}}K(U(s-))y(U(s-))\mathrm{d}s = u_{0} + \int_{0}^{\tau}K(U(s-))y(U(s-))\mathrm{d}s. 
	\end{aligned}
\end{eqnarray*}
Apply the standard Gronwall inequality path by path to yield
\begin{eqnarray*}
   x(U(-\tau)) \leq y(U(\tau-)) \leq u_{0}\exp\bigg(\int_{0}^{\tau}K(U(s-))\mathrm{d}s\bigg).
\end{eqnarray*}
For every $t\in[0, ~T]$, let $\tau = E_{t}$. Then, applying first the relation in Eq.~\eqref{relationofinverseandsubordinator} followed by Lemma~\ref{1stchangeofvariable}
\begin{eqnarray*}
	\begin{aligned}
	  	x(t) &\leq  y(t) \leq y(U(E_{t}-))\leq u_{0}\exp\bigg(\int_{0}^{E_{t}} K(U(s-))\mathrm{d}s\bigg)= u_{0}\exp\bigg(\int_{0}^{t}K(s)\mathrm{d}E_{s}\bigg),
	\end{aligned}
\end{eqnarray*}
thereby completing the proof.
\end{proof}

\begin{lem}\label{timechangeditoformula}
Suppose $U(t)$ is a $\beta$-stable subordinator and $E_{t}$ is the associated inverse stable subordinator. Define a filtration $\{\mathcal{G}_{t}\}_{t\geq 0}$ by $\mathcal{G}_t = \mathcal{F}_{E_t}$ where $\mathcal{F}_{t}$ is the filtration defined in Eq.~\eqref{filtrationgeneratedbyBandE}. Let $X(t)$ be a process defined by the following time-changed process
\begin{eqnarray*}
  X(t) = x_{0} + \int_{0}^{t}P(s)\mathrm{d}s + \int_{0}^{t}\Phi(s)\mathrm{d}E_{s} + \int_{0}^{t}\Psi(s)\mathrm{d}B_{E_s},
\end{eqnarray*}
where $P, \Phi$ and $\Psi$ are measurable functions such that all integrals are defined. If $F: {R_{+}}\times {R_{+}}\times {R^{n}} \to {R}$ is a $C^{1, 1, 2}({R_{+}}\times {R_{+}}\times {R^{n}} ; {R})$ function, then with probability one
\begin{eqnarray*}
	\begin{aligned}
		F(t, E_{t}, X(t)) - F(0, 0, x_{0}) &= \int_{0}^{t}F_{t_{1}}(t, E_{s}, X(s))\mathrm{d}s + \int_{0}^{t}F_{t_{2}}(s, E_{s}, X(s))\mathrm{d}E_{s}\\ 
		&+ \int_{0}^{t}F_{x}(s, E_{s}, X(s))P(s)\mathrm{d}s + \int_{0}^{t}F_{x}(s, E_{s}, X(s))\Phi(s)\mathrm{d}E_{s}\\
		& +\int_{0}^{t}F_{x}(s, E_{s}, X(s))\Psi(s)\mathrm{d}B_{E_{s}} + \frac{1}{2}\int_{0}^{t}\Psi^{T}(s)F_{xx}(s, E_s, X(s))\Psi(s)\mathrm{d}E_s,
	\end{aligned}
\end{eqnarray*} 
where $F_{t_{1}}$, $F_{t_{2}}$ and $F_{x}$ are first derivatives, respectively, and $F_{xx}$ denotes the second derivative.
\end{lem}
\begin{proof}
Let $Y(t) := \begin{bmatrix}
	       t\\
	       E_t\\
	       X(t)\\
\end{bmatrix}$. Then, the stochastic process $Y(t)$ is defined as 
\begin{eqnarray*}
    Y_t = \begin{bmatrix}
		t\\
		E_t\\
		x_{0} + \int_{0}^{t}P(s)\mathrm{d}s + \int_{0}^{t}\Phi(s)\mathrm{d}E_{s} + \int_{0}^{t}\Psi(s)\mathrm{d}B_{E_s}
	\end{bmatrix}.
\end{eqnarray*}
Let $y = \begin{bmatrix}
t_{1}\\
t_{2}\\ 
x\end{bmatrix}$ and $G(y) = F(t_{1}, t_{2},  x)$ which is twice differentiable in $x$ and first differentiable in $t_{1}$ and $t_{2}$. Based on the computation rules
\begin{eqnarray}\label{computationrules}
	\mathrm{d}t\cdot\mathrm{d}t = \mathrm{d}E_{t}\cdot\mathrm{d}E_{t} =   \mathrm{d}t\cdot\mathrm{d}E_{t} = \mathrm{d}t\cdot\mathrm{d}B_{E_{t}} = \mathrm{d}E_{t}\cdot\mathrm{d}B_{E_{t}} = 0,~~~ \mathrm{d}B_{E_{t}}\cdot\mathrm{d}B_{E_{t}} = \mathrm{d}E_{t},
\end{eqnarray} 
apply the standard multi-dimensional It$\hat{o}$ formula to $G(y)$ to obtain
\begin{eqnarray*}
	\begin{aligned}
		\mathrm{d}G(Y(t)) &= G_{y}(Y(t))\mathrm{d}Y(t) + \frac{1}{2}\mathrm{d}Y(t)^{T}G_{yy}(Y(t))\mathrm{d}Y(t)\\
		&= \begin{bmatrix}F_{t_{1}}(t, E_t, X(t))~~ F_{t_{2}}(t, E_{t}, X(t)) ~~F_{x}(t, E_t, X(t))\end{bmatrix}\begin{bmatrix}
			\mathrm{}dt\\
			\mathrm{d}E_t\\
			P(t)\mathrm{d}t + \Phi(t)\mathrm{d}E_{t} + \Psi(t)\mathrm{d}B_{E_t}
		\end{bmatrix}\\
		&~~~+ \frac{1}{2}\Psi^{T}(t)F_{xx}(t, E_t, X(t))\Psi(t)\mathrm{d}E_t\\
		&= F_{t_{1}}(t, E_{t}, X(t))\mathrm{d}t + F_{t_2}(t, E_t, X(t))\mathrm{d}E_{t} + F_{x}(t, E_{t}, X(t))P(t)\mathrm{d}t + F_{x}(t, E_t, X(t))\Phi(t)\mathrm{d}E_t\\	
		& + F_{x}(t, E_t, X(t))\Psi(t)\mathrm{d}B_{E_t} + \frac{1}{2}\Psi^{T}(t)F_{xx}(t, E_t, X(t))\Psi(t)\mathrm{d}E_t.           
	\end{aligned}
\end{eqnarray*}
Although the second derivative of function $F(t_{1}, t_{2}, x)$ with respect to $t_{1}$ and $t_{2}$ may not exist, according to computation rules Eq.~\eqref{computationrules}, the above application of the standard multi-dimensional It\^o formula for continuous semimartingale process still works. Then,
\begin{eqnarray*}
	\begin{aligned}
		F(t, E_{t}, X(t)) &- F(0, 0, x_{0}) = \int_{0}^{t}\bigg\{F_{t_{1}}(s, E_{s}, X(s)) + F_{x}(s, E_{s}, X(s))P(s)\bigg\}\mathrm{d}s\\
		&+  \int_{0}^{t}\bigg\{F_{t_{2}}(s, E_{s}, X(s)) + F_{x}(s, E_s, X(s))\Phi(s) + \frac{1}{2}\Psi^{T}(s)F_{xx}(s, E_s, X(s))\Psi(s)\bigg\}\mathrm{d}E_{s}\\
		&+ \int_{0}^{t}F_{x}(s, E_s, X(s))\Psi(s)\mathrm{d}B_{E_s},
	\end{aligned}
\end{eqnarray*}
which is the desired result.
\end{proof}

After establishing the time-changed Gronwall's inequality and the generalized time-changed It\^o formula, the first type of time-changed differential system we considered is 
\begin{eqnarray}\label{timechangednonlinearsystem}
\left\{ \begin{array}{rl}
&\mathrm{d}X(t) = AX(t)\mathrm{d}{E_{t}} + f(E_{t}, X(t))\mathrm{d}{E_{t}}\\
& X(0) = x_{0},
\end{array} \right.
\end{eqnarray}
where $A$ is a deterministic matrix. The corresponding non-time-changed system is
\begin{eqnarray}\label{classicnonlinearsystem}
\left\{ \begin{array}{rl}
&\mathrm{d}Y(t) = AY(t) \mathrm{d}t + f(t, Y(t))\mathrm{d}{t}\\
& Y(0) = x_{0},
\end{array} \right.
\end{eqnarray}
which plays an important role in applied science and engineering. The time-changed system, Eq.~\eqref{timechangednonlinearsystem}, occurs when the system evolves only during the operation time  $E_{t}$.
\begin{thm}\label{thm1}
Let $A$ be an $n\times n$ real constant matrix with $\mathrm{Re}(\sigma(A)) < 0$. Suppose $f: R^{+}\times R^{n}\to R^{n}$ is a nonlinear function which satisfies
\begin{eqnarray}\label{boundoffunctionf}
	\|f(E_{t}, X(t))\| \leq \|g(E_{t})\|\|X(t)\|
\end{eqnarray}
with the function $g: R^{+}\to R^{n}$ satisfying
\begin{eqnarray}\label{integrabilityofg}
	\int_{0}^{\infty}\|g(s)\|\mathrm{d}s < \infty.
\end{eqnarray}
Then the trivial solution of the time-changed nonlinear system, Eq.~\eqref{timechangednonlinearsystem}, is exponentially sample-path stable and $p$th moment asymptotically stable.
\end{thm}
\begin{proof}
Let $F(t_{1}, t_{2}, x) = \exp(t_{2})x$. Apply the time-changed It\^o formula, Lemma~\ref{timechangeditoformula}, to the time-change system, Eq.~\eqref{timechangednonlinearsystem}, to yield 
\begin{eqnarray}\label{expressionofX(t)afteritoformula}
	X(t) = \exp(AE_{t})x_{0} + \int_{0}^{t}\exp(A(E_{t}-E_{s}))f(E_{s}, X(s))\mathrm{d}E_{s}.
\end{eqnarray}
Since $\mathrm{Re}(\sigma(A)) < 0$, there is a constant $K > 0$ and $\lambda > 0$ such that, for all $t > 0$, 
\begin{eqnarray}\label{boundofmatrixA}
	\|\exp(At)\| \leq K\exp(-\lambda t).
\end{eqnarray}
Taking the norm on both sides of Eq.~\eqref{expressionofX(t)afteritoformula} and applying conditions, Eqs.~\eqref{boundoffunctionf} and~\eqref{boundofmatrixA}, yields
\begin{eqnarray*}
	\begin{aligned}
		\|X(t)\| &\leq K\exp(-\lambda E_{t})\| x_{0}\|  + \int_{0}^{t}K\exp(-\lambda(E_{t}-E_{s}))\| f(E_{s}, X(s))\|\mathrm{d}E_{s}\\
		&\leq K\exp(-\lambda E_{t})\|x_{0}\|  + \int_{0}^{t}K\exp(-\lambda(E_{t}-E_{s}))\|g(E_{s})\|\|X(s)\|\mathrm{d}E_{s}.
	\end{aligned}
\end{eqnarray*}
This means
\begin{eqnarray*}
	\exp(\lambda E_{t})\| X(t)\| \leq K\|x_{0}\| + K\int_{0}^{t}\|g(E_{s})\|\exp(\lambda E_{s})\|X(s)\|\mathrm{d}E_{s}.
\end{eqnarray*}
Apply the time-changed Gronwall's inequality, Lemma \ref{gronwall}, to yield almost surely 
\begin{eqnarray*}
	\exp(\lambda E_{t})\|X(t)\| \leq K\|x_{0}\|\exp\bigg(K\int_{0}^{t}\| g(E_{s})\|\mathrm{d}E_{s}\bigg),
\end{eqnarray*}
which implies almost surely
\begin{eqnarray}\label{boundofnormofX(t)}
     \|X(t)\| \leq \exp(-\lambda E_{t}) K\|x_{0}\|\exp\bigg(K\int_{0}^{t}\| g(E_{s})\|\mathrm{d}E_{s}\bigg).
\end{eqnarray}
Combine Lemma~\ref{1stchangeofvariable} and condition Eq.~\eqref{integrabilityofg}
  to yield 
\begin{eqnarray}\label{boundofintegralofg(t)}
	\int_{0}^{t}\lVert g(E_{s})\rVert \mathrm{d}E_{s} = \int_{0}^{E_{t}}\lVert g(s)\rVert \mathrm{d}s \leq \int_{0}^{\infty}\lVert g(s)\rVert \mathrm{d}s < \infty.
\end{eqnarray}
Also since $E_{t} \to \infty$ as $t\to\infty$ almost surely, it indicates from Eq.s~\eqref{boundofnormofX(t)} and~\eqref{boundofintegralofg(t)} that $\|X(t)\|\to 0$ exponentially in the sense of almost sure convergence. Moreover, from Eq.~\eqref{boundofnormofX(t)},
\begin{eqnarray*}
	\mathbb{E}\|X(t)\|^{p}\leq \mathbb{E}\bigg\{\exp\big(-\lambda pE_{t}\big)K^{p}\| x_{0}\|^{p}\exp\bigg(Kp\int_{0}^{t}\| f(E_{s})\|\mathrm{d}E_{s}\bigg)\bigg\}.
\end{eqnarray*}
Again from Lemma~\ref{1stchangeofvariable} and the fact that $E_{t} \to \infty$ as $t\to \infty$ almost surely,
\begin{eqnarray}\label{pthmoment}
\mathbb{E}\|X(t)\|^{p} \leq \mathbb{E}\bigg\{\exp\big(-\lambda pE_{t}\big)\bigg\}K^{p}\| x_{0}\|^{p}\exp\bigg(Kp\int_{0}^{\infty}\lVert f(s)\rVert\mathrm{d}s\bigg).
\end{eqnarray}
On the other hand, the inverse $\beta$-stable subordinator $E_{t}$ takes Laplace transform
\begin{eqnarray}\label{laplacetransformofE_{t}}
   \mathbb{E}(\exp(-\lambda E_{t})) = E_{\beta}(-\lambda t^{\beta}) ,
\end{eqnarray}
where $E_{\beta}(t)$ is the Mittag-Leffler function defined by $E_{\beta}(t) = \sum_{k=0}^{\infty}\frac{t^{k}}{\Gamma(k\beta + 1)}$ with Gamma function $\Gamma(t)$ for $t\geq 0$. Also $E_{\beta}(-\lambda t^{\beta})\to 0$ as $t\to\infty$, see \cite{mainardi2013some}. 
Then, from Eq.s~\eqref{pthmoment} and~\eqref{laplacetransformofE_{t}},
\begin{eqnarray*}
	\mathbb{E}\|X(t)\|^{p} \leq E_{\beta}(-\lambda pt^{\beta})K^{p}\| x_{0}\|^{p}\exp\bigg(Kp\int_{0}^{\infty}\lVert f(s)\rVert\mathrm{d}s\bigg)\to 0.
\end{eqnarray*}
Therefore, the trivial solution $X(t)$ of the time-changed system Eq.~\eqref{timechangednonlinearsystem} is exponentially sample-path stable and $p$th moment asymptotically stable.
\end{proof}

\begin{corl}\label{corl1}
Let $A$ be an $n\times n$ real constant matrix with $\mathrm{Re}(\sigma(A)) < 0$. Suppose $f: R^{+}\times R^{n}\to R^{n}$ is a nonlinear function. If the trivial solution of the non-time-changed system Eq.~\eqref{classicnonlinearsystem} is exponentially stable,
then the trivial solution of the time-changed system Eq.~\eqref{timechangednonlinearsystem} is $p$th moment asymptotically stable.
\end{corl}

\begin{proof}
Let $Y(t)$ be the solution of the non-time-changed system Eq.~\eqref{classicnonlinearsystem}. By the duality Lemma \ref{duality}, the process $X(t) := Y(E_{t})$ is the solution of time-changed system Eq.~\eqref{timechangednonlinearsystem}. Also since the solution, $Y(t)$, of the non-time-changed system Eq.~\eqref{classicnonlinearsystem}  is exponentially stable, there exists positive constants, $K$ and $\lambda$, such that
\begin{eqnarray*}
	\|Y(t)\| \leq K\exp(-\lambda t).
\end{eqnarray*}
Applying conditional expectation yields
\begin{eqnarray*}
	\begin{aligned}
		\mathbb{E}\|X(t)\|^{p} &= \mathbb{E}\|Y(E_{t})\|^{p} = \int_{0}^{\infty}\mathbb{E}\bigg(\|Y(E_{t})\|^{p}\bigg|E_{t} = \tau\bigg)f_{E_{t}}(\tau)\mathrm{d}\tau = \int_{0}^{\infty}\|Y(\tau)\|^{p}f_{E_{t}}(\tau)\mathrm{d}\tau\\
		&\leq \int_{0}^{\infty}K^{p}\|x_{0}\|^{p}\exp(-p\lambda \tau)f_{E_{t}}(\tau)\mathrm{d}\tau = K^{p}\|x_{0}\|^{p}\mathbb{E}(-p\lambda E_{t}) =  K^{p}\|x_{0}\|^{p}E_{\beta}(-p\lambda t^{\beta}).
	\end{aligned}
\end{eqnarray*} 
Therefore, the trivial solution of time-changed system Eq.~\eqref{timechangednonlinearsystem} is $p$th moment asymptotically stable.
\end{proof}

\begin{rem}Theorem \ref{thm1} indicates that although the sample path of the trivial solution of the time-changed nonlinear system Eq.~\eqref{timechangednonlinearsystem} is exponentially stable, the $p$th $(p\geq 1)$ moment of the trivial solution is asymptotically stable. This makes sense because the inverse $\beta$-stable subordinator, $E_{t}$, has a distribution with a heavy tail. The long range dependence (i.e. memory) will slow the decay rate of the $p$-th moment even though every sample path decays exponentially.  
\end{rem}

\begin{rem}
Actually, under conditions Eq.s~\eqref{boundoffunctionf} and~\eqref{integrabilityofg}, the trivial solution of the non-time-changed system Eq.~\eqref{classicnonlinearsystem} is exponentially stable. In this sense, Corollary \ref{corl1} is directly derived from Theorem \ref{thm1}. However, based on the duality Lemma \ref{duality}, Corollary \ref{corl1} provides a deep connection on stability between the non-time-changed system Eq.~\eqref{classicnonlinearsystem} and the time-changed system Eq.~\eqref{timechangednonlinearsystem}.
\end{rem}

The next time-changed system can be considered as a perturbed version of a linear system. However, the external force term is affected by the operation time $E_{t}$. So the perturbed time-changed system is
\begin{eqnarray}\label{mixedtimechangedsystem}
	\left\{ \begin{array}{rl}
		&\mathrm{d}X_{t} = AX_{t}\mathrm{d}t + f(E_{t}, X_{t})\mathrm{d}{E_{t}}\\
		& X(0) = x_{0}.
	\end{array} \right.
\end{eqnarray}

\begin{thm}\label{thmonmixedtimechangedsystem}
	Let $A$ be an $n\times n$ real constant matrix with $\mathrm{Re}(\sigma(A)) < 0$. Suppose $f: R^{+}\times R^{n}\to R^{n}$ is a nonlinear function which satisfies conditions Eq.s~\eqref{boundoffunctionf} and~\eqref{integrabilityofg}. Then the trivial solution of the time-changed system Eq.~\eqref{mixedtimechangedsystem} is sample-path and $p$th moment exponentially stable.
\end{thm}
		
\begin{proof}
Let $F(t_{1}, x) = \exp(-t_{1}A)x$. Apply the time-changed It\^o Lemma~\ref{timechangeditoformula} to the time-changed system Eq.~\eqref{mixedtimechangedsystem} to yield
\begin{eqnarray*}
    X(t) = \exp(At)x_{0} + \int_{0}^{t}\exp(A(t - s))f(E_{s},  X(s))\mathrm{d}E_{s}.
\end{eqnarray*}
Applying the condition Eq.~\eqref{boundoffunctionf} and the fact that $\mathrm{Re}(\sigma(A)) < 0$ yields
\begin{eqnarray*}
	\|X(t)\| \leq K\exp(-\lambda t)\lVert x_{0}\rVert + K\int_{0}^{t}\exp(-\lambda (t-s))\lVert g(E_{s})\rVert \lVert X_{s}\rVert\mathrm{d}E_{s}.
\end{eqnarray*}
From Gronwall's inequality of Lemma~\ref{gronwall} and the first change of variable Lemma~\eqref{1stchangeofvariable},
\begin{eqnarray}\label{boundofnormofx(t)ofmixedtimechangedsystem}
	\begin{aligned}
		\|X(t)\| &\leq \exp(-\lambda t)K\| x_{0}\|\exp\bigg(K\int_{0}^{t}\| g(E_{s})\|\mathrm{d}E_{s}\bigg)\\
		& = \exp(-\lambda t)K\| x_{0}\|\exp\bigg(K\int_{0}^{E_{t}}\lVert g(s)\rVert \mathrm{d}s\bigg).
	\end{aligned}
\end{eqnarray}
Similarly, applying the finiteness condition, Eq.~\eqref{integrabilityofg}, to Eq.~\eqref{boundofnormofx(t)ofmixedtimechangedsystem} yields $\|X_{t}\|\to 0$ exponentially for every sample path as $t\to \infty$. This means the trivial solution of the time-changed system Eq.~\eqref{mixedtimechangedsystem} is sample-path exponentially stable. Moreover, from Eq.~\eqref{boundofnormofx(t)ofmixedtimechangedsystem}, 
\begin{eqnarray*}
	\begin{aligned}
		\mathbb{E}\|X(t)\|^{p} &\leq \exp(-p\lambda t)K^{p}\| x_{0}\|^{p}\mathbb{E}\bigg\{\exp\bigg(Kp\int_{0}^{E_t}\|g(s)\| \mathrm{d}s\bigg)\bigg\}\\
		&\leq \exp(-p\lambda t)K^{p}\| x_{0}\|^{p}\exp\bigg(Kp\int_{0}^{\infty}\lVert g(s)\rVert \mathrm{d}s\bigg).
	\end{aligned}
\end{eqnarray*}
Therefore, $\mathbb{E}\|X(t)\|^{p} \to 0$ exponentially which means the trivial solution of the time-changed system~\eqref{mixedtimechangedsystem} is also $p$th moment exponentially stable.		
\end{proof}

\begin{rem}
Theorem~\ref{thmonmixedtimechangedsystem} reveals that although the linear system is disturbed by the environment which incorporates long-term memory dependent behavior, the trivial solution of the disturbed system Eq.~\eqref{mixedtimechangedsystem} is both sample-path and $p$th moment exponentially stable. This stability of the system Eq.~\eqref{mixedtimechangedsystem} is different from the stability of the system Eq.~\eqref{timechangednonlinearsystem}. This difference results from whether or not the dominant part of the linear system is affected by the operation time $E_{t}$.
\end{rem}
Finally, consider the time-changed system which can be considered as a time-changed linear system perturbed by long-term memory dependent noise with the noise being the time-changed Brownian motion $B_{E_{t}}$. 
\begin{eqnarray}\label{stochastictimechangednonlinearsystem}
\left\{\begin{array}{rl}
&\mathrm{d}X(t) = AX(t)\mathrm{d}E_{t} +  f(E_{t}, X(t))\mathrm{d}B_{E_{t}}\\
& X(0) = x_{0},
\end{array} \right.
\end{eqnarray}
where $B_{t}$ is a standard Brownian motion. 
\begin{thm}
	Let $A$ be an $n\times n$ real constant matrix with $\mathrm{Re}(\sigma(A)) < 0$. Supose $f: R^{+}\times R^{n}\to R^{n}$ is a nonlinear function which satisfies condition Eq.~\eqref{boundoffunctionf} and a function $g: R^{+}\to R$ which satisfies
	\begin{eqnarray}\label{integrabilityofgg}
	\int_{0}^{\infty}\|g(s)\|^{2}\mathrm{d}s < \infty.
	\end{eqnarray} Then the trivial solution of the time-changed system Eq.~\eqref{mixedtimechangedsystem} is square-mean asymptotically stable.  
\end{thm}
\begin{proof}
Suppose the following non-time-changed stochastic differential system  corresponds to the time-changed system Eq.~\eqref{stochastictimechangednonlinearsystem}
\begin{eqnarray}\label{stochasticnonlinearsystem}
\left\{\begin{array}{rl}
&\mathrm{d}Y(t) = A(t)\mathrm{d}t +  f(t, Y(t))\mathrm{d}B_{t}\\
& Y(0) = x_{0}.
\end{array} \right.
\end{eqnarray}
Let $F(t, y) = \exp(At)y$. Applying the standard It\^o formula to Eq.~\eqref{stochasticnonlinearsystem} yields 
\begin{eqnarray}\label{solutionoftimechangedSDE}
	\begin{aligned}
		Y(t) &= \exp(At)x_{0} + \int_{0}^{t}\exp(A(t - s))f(s, Y(s))\mathrm{d}B_{s}.
	\end{aligned}
\end{eqnarray}
It is known from \cite{Marcin2010} and \cite{kuo05} that
\begin{eqnarray*}
	\int_{0}^{t}\exp(A(t - s))f(s, Y(s))\mathrm{d}B_{s}
\end{eqnarray*}
is a square integrable martingale. So apply the Cauchy inequality and  It\^o identity to yield
\begin{eqnarray*}
	\begin{aligned}
		\mathbb{E}\|Y(t))\|^{2} &\leq 2\mathbb{E}\| \exp(At)x_{0}\|^{2} +  2\mathbb{E}\bigg\| \int_{0}^{t}\exp(A(t - s)f(s, Y(s))\mathrm{d}B_{s}\bigg\|^{2}\\
		&\leq 2\|\exp(At)\|^{2}\|x_{0}\|^{2} + 2\mathbb{E}\bigg(\int_{0}^{t}\|\exp(A(t - s))\|^{2}\|f(s, Y(s))\|^{2}\mathrm{d}s\bigg).
	\end{aligned}
\end{eqnarray*}
Since $Re(\sigma(A)) < 0$ and the nonlinear function $f$ satisfies conditions, Eq.~\eqref{boundoffunctionf} and~\eqref{integrabilityofgg}, 
\begin{eqnarray*}
	\mathbb{E}\|Y(t)\|^{2}	&\leq 2K^{2}\exp(-2\lambda t) \|x_{0}\|^{2} + 2K^{2}\int_{0}^{t}\exp(-2\lambda(t - s))\|g(s)\|^{2}\mathbb{E}\|Y(s)\|^{2}\mathrm{d}s.
\end{eqnarray*}
Using the standard Gronwall's inequality yields
\begin{eqnarray}\label{boundofmeansquareofnontimechangesolution}
	\begin{aligned}
		\mathbb{E}\|Y(t)\|^{2}	&\leq 2K^{2}\exp(-2\lambda t) \|x_{0}\|^{2}\exp\bigg(2K^{2}\int_{0}^{t}\|g(s)\|^{2}\mathrm{d}s\bigg),
	\end{aligned}
\end{eqnarray}
which results in $\mathbb{E}\|Y\|^{2}\to 0$ exponentially 
from condition Eq.~\eqref{integrabilityofgg}. Moreover, let $X(t) := Y(E_{t})$ and then $X(t)$ is the solution of the stochastic time-changed system Eq.~\eqref{stochastictimechangednonlinearsystem} from duality Theorem~\ref{duality}. Then, combining conditional expectation with Eq.~\eqref{boundofmeansquareofnontimechangesolution} yields
\begin{eqnarray*}
	\begin{aligned}
		\mathbb{E}\|X(t)\|^{2} &= \mathbb{E}\|Y(E_{t})\|^{2} = \int_{0}^{\infty}\mathbb{E}\bigg(\|Y(E_{t})\|^{2}\bigg|E_{t} = \tau\bigg)f_{E_{t}}(\tau)\mathrm{d}\tau\\
		&\leq K_{0}\int_{0}^{\infty}\exp(-2\lambda\tau)f_{E_{t}}(\tau)\mathrm{d}\tau = K_{0}\mathrm{E}(\exp(-2\lambda E_{t})) = K_{0}E_{\beta}(-2\lambda t^{\beta}),
	\end{aligned}
\end{eqnarray*}
where $K_{0} = 2K^{2}\exp\bigg(2K^{2}\int_{0}^{\infty}\|g(s)\|^{2}\mathrm{d}s\bigg)$. Therefore the trivial solution of the time-changed system, Eq.~\eqref{stochastictimechangednonlinearsystem}, is square-mean asymptotically stable.
\end{proof}

\section{Acknowledgments}
The author wishes to thank Dr. Marjorie Hahn for her advice, encouragement and patience with my research, and the author's peer Lise Chlebak as well as Dr. Patricia Garmirian for their discussions. 

%\section*{References}
\bibliography{mybibfile}

\begin{thebibliography}{10}
\expandafter\ifx\csname url\endcsname\relax
  \def\url#1{\texttt{#1}}\fi
\expandafter\ifx\csname urlprefix\endcsname\relax\def\urlprefix{URL }\fi
\expandafter\ifx\csname href\endcsname\relax
  \def\href#1#2{#2} \def\path#1{#1}\fi

\bibitem{Rugh1996}
W.~J. Rugh, Linear system theory, Vol.~2, prentice hall Upper Saddle River, NJ,
  1996.

\bibitem{Feng1992}
X.~Feng, K.~Loparo, Y.~Ji, H.~Chizeck, Stochastic stability properties of jump
  linear systems, Automatic Control, IEEE Transactions on 37~(1) (1992) 38--53.

\bibitem{erneux2009applied}
T.~Erneux, Applied delay differential equations, Vol.~3, Springer Science \&
  Business Media, 2009.

\bibitem{matignon1996stability}
D.~Matignon, Stability results for fractional differential equations with
  applications to control processing, in: Computational engineering in systems
  applications, Vol.~2, Lille France, 1996, pp. 963--968.

\bibitem{Kei2011}
K.~Kobayashi, Stochastic calculus for a time-changed semimartingale and the
  associated stochastic differential equations, Journal of Theoretical
  Probability 24~(3) (2011) 789--820.

\bibitem{Janicki1994}
A.~Janicki, A.~Weron, Simulation and Chaotic Behavior of Alpha-stable
  Stochastic Processes, Hugo Steinhaus Center, Wroclaw University of
  Technology, 1994.

\bibitem{Mircea2002}
M.~Grigoriu, Stochastic Calculus: Application in Science and Engineering,
  Springer Science \& Business Media, 2002.

\bibitem{Marcin2010}
M.~Magdziarz, Path properties of subdiffusion - a martingale approach,
  Stochastic Models 26~(2) (2010) 256--271.

\bibitem{Jean1979}
J.~Jacod, Calcul Stochastique et Probl\`{e}mes de Martingales, Lecture Notes in
  Mathematics, 714. Springer, Berlin, 1979.

\bibitem{mainardi2013some}
F.~Mainardi, On some properties of the {Mittag-Leffler} function
  {$E_{\alpha}(-t^{\alpha})$}, completely monotone for $t> 0$ with $0 < \alpha
  < 1$, arXiv preprint arXiv:1305.0161.

\bibitem{kuo05}
H.-H. Kuo, {Introduction to stochastic integration}, Springer, 2005.

\end{thebibliography}

\end{document}